\documentclass[11pt,reqno]{article}

\usepackage{amsmath,amsthm,amsfonts,amssymb,latexsym,mathrsfs,color}
\usepackage{hyperref}

\newtheorem{theorem}{Theorem}

\newtheorem{proposition}[theorem]{Proposition}

\newcommand{\mbn}{{\mathcal B}_n}

\newcommand{\des}{{\rm des\,}}
\newcommand{\msn}{{\mathcal S}_n}
\newcommand{\pn}{{\mathcal P}(n,n-r)}
\newcommand{\lrf}[1]{\lfloor #1\rfloor}

\newcommand{\stirling}[2]{\genfrac{[}{]}{0pt}{}{#1}{#2}}
\newcommand{\Stirling}[2]{\genfrac{\{}{\}}{0pt}{}{#1}{#2}}
\newcommand{\Eulerian}[2]{\genfrac{<}{>}{0pt}{}{#1}{#2}}
\newcommand{\lah}[2]{\genfrac{\lfloor}{\rfloor}{0pt}{}{#1}{#2}}
\title{Some combinatorial sequences associated with context-free grammars\footnote{This work is supported by~NSFC (11126217) and the Fundamental Research Funds for the Central Universities (N100323013).}}

\author
{Shi-Mei Ma \footnote{ {\it Email address:}
shimeima@yahoo.com.cn (S.-M. Ma)} }
\date{\footnotesize School of Mathematics and Statistics,
        Northeastern University at Qinhuangdao,\\ Hebei 066004,
        China}

\begin{document}

\maketitle

\begin{abstract}
The purpose of this paper is to show that some combinatorial sequences, such as second-order Eulerian numbers and Eulerian numbers of type $B$, can be generated by context-free grammars.
\end{abstract}
\section{Introduction}\label{sec:intro}
The grammatical method was introduced by Chen~\cite{Chen93} in the study of exponential structures in combinatorics. Let $A$ be an alphabet whose letters are regarded as independent commutative indeterminates. A context-free grammar $G$ over $A$ is defined as a set of substitution rules replacing a letter in
$A$ by a formal function over $A$.
Following Chen~\cite{Chen93}, the formal derivative $D$ is a linear operator defined with respect to a context-free grammar $G$.
For any formal functions $u$ and $v$, we have
$$D(u+v)=D(u)+D(v),\quad D(uv)=D(u)v+uD(v) \quad and\quad D(f(u))=\frac{\partial f(u)}{\partial u}D(u),$$
where $f(x)$ is a analytic function. By definition, we have $D^{n+1}(u)=D(D^n(u))$ for all $u$.
For example, if $G=\{x\rightarrow xy, y\rightarrow y\}$, then $$D(x)=xy,D(y)=y,D^2(x)=x(y+y^2),D^3(x)=x(y+3y^2+y^3).$$
In~\cite{Dumont96}, Dumont considered chains of general substitution rules on words.
It is a hot topic to explore the connection between combinatorics and context-free grammars. The reader is referred to~\cite{Chen121,Chen122,Dumont962,Ma12} for recent
progress on this subject.

We now recall some definitions, and fix some notation, that will be used throughout the
rest of this paper.
Let $[n]=\{1,2,\ldots,n\}$.
Let $\msn$ denote the symmetric group of all permutations of $[n]$.
The {\it Eulerian number} $\Eulerian{n}{k}$ enumerates the number of permutations in $\msn$
with $k$ descents (i.e., $i<n,\pi(i)>\pi(i+1)$) as well as the number permutations in $\msn$ which have $k$ excedances (i.e., $i<n,\pi(i)>i)$) (see~\cite[A008292]{Sloane}). The numbers $\Eulerian{n}{k}$
satisfy the recurrence relation
$$\Eulerian{n}{k}=(k+1)\Eulerian{n-1}{k}+(n-k)\Eulerian{n-1}{k-1},$$
the initial condition $\Eulerian{0}{0}=1$ and boundary conditions $\Eulerian{0}{k}=0$ for $k\geq 1$.
Let $$A_n(t)=\sum_{k=0}^{n-1}\Eulerian{n}{k}t^k$$ be the {\it Eulerian polynomial}.
The exponential generating function for $A_n(t)$ is
\begin{equation}\label{Atz}
A(t,z)=1+\sum_{n\geq 1}tA_n(t)\frac{z^n}{n!}=\frac{1-t}{1-te^{z(1-t)}}.
\end{equation}

We now consider a restricted version of Eulerian numbers.
Let $r$ be a nonnegative integer. Denote by $\pn$ the set of permutations of $n$ numbers taken $n-r$ at a time.
Let $\sigma\in\pn$. If $\sigma(i)>i$, then we say that $\sigma$ has an excedance at position $i$, where $1\leq i\leq n-r$.
The {\it $r$-restricted Eulerian number}, denoted by $\Eulerian{n}{k}_{r}$, is defined as the number of permutations in $\pn$ having $k$ excedances (see~\cite[A144696 ,~A144697,~A144698,~A144699]{Sloane} for details).

A {\it Stirling permutation} of order $n$ is a permutation of the multiset $\{1,1,2,2,\ldots,n,n\}$ such that for each $i$, $1\leq i\leq n$, the elements lying between the two occurrences of $i$ are greater than $i$.
{\it The second-order Eulerian number} $\left < \!\! \left < {n \atop k} \right > \!\! \right >$ is
the number of  Stirling permutation of order $n$ with
$k$ ascents (see~\cite[A008517]{Sloane}).
The combinatorial interpretations for the second-order Eulerian numbers $\left < \!\! \left < {n \atop k} \right > \!\! \right >$ have been
extensively investigated (see~\cite{Bona08,Gessel78,Haglund12}).
It is well known that the numbers $\left < \!\! \left < {n \atop k} \right > \!\! \right >$ satisfy the recurrence relation
\begin{equation}\label{Eu-second-recu}
\left < \!\! \left < {n+1 \atop k} \right > \!\! \right >=(2n-k+1)\left < \!\! \left < {n \atop k-1} \right > \!\! \right >+(k+1)\left < \!\! \left < {n \atop k} \right > \!\! \right >,
\end{equation}
with initial condition $\left < \!\! \left < {1 \atop 0} \right > \!\! \right >=1$ and boundary conditions $\left < \!\! \left < {n \atop k} \right > \!\! \right >=0$ for $n\leq k$ or $k<0$ (see~\cite[A008517]{Sloane}).

Let $B_n$ denote the set of signed permutations of $\pm[n]$ such that $\pi(-i)=-\pi(i)$ for all $i$, where $\pm[n]=\{\pm1,\pm2,\ldots,\pm n\}$.
Let
$${B}_n(x)=\sum_{k=0}^nB(n,k)x^{k}=\sum_{\pi\in \mbn}x^{\des_B(\pi)},$$
where
$$\des_B=|\{i\in[n]:\pi(i-1)>\pi({i})\}|$$
with $\pi(0)=0$.
The polynomial $B_n(x)$ is called an {\it Eulerian polynomial of type $B$}, while $B(n,k)$ is called an {\it Eulerian number of type $B$} (see~\cite[A060187]{Sloane}).
The first few of these polynomials are listed below:
$$B_0(x)=1,B_1(x)=1+x,B_2(x)=1+6x+x^2,B_3(x)=1+23x+23x^2+x^3.$$
The numbers $B(n,k)$ satisfy
the recurrence relation
\begin{equation}\label{Bnk-Euleriannum}
B(n+1,k)=(2n-2k+3)B(n,k-1)+(2k+1)B(n,k),
\end{equation}
with initial condition $B(0,0)=1$ and boundary conditions $B(0,k)=0$ for $k\geq 1$.
An explicit formula for $B(n,k)$ is given as follows:
\begin{equation*}
B(n,k)=\sum_{i=0}^k(-1)^{i}\binom{n+1}{i}(2k-2i+1)^{n}
\end{equation*}
for $0\leq k\leq n$ (see~\cite{Eriksen00} for details).

{\it The unsigned Stirling number of the first kind} $\stirling{n}{k}$ is the number of permutations in $\msn$ with exactly $k$ cycles (see~\cite[A132393]{Sloane}). {\it The Stirling number of the second kind}
$\Stirling{n}{k}$ is the number of ways to partition $[n]$ into $k$ blocks (see~\cite[A008277]{Sloane}).
Let $\lah{n}{k}$ denote the number of ways to partition $[n]$ into $k$ nonempty linearly ordered subsets. The numbers $\lah{n}{k}$ are called {\it the unsigned Lah numbers} (see~\cite[A105278]{Sloane}).

We recall some known results on context-free grammars.
\begin{proposition}[{\cite[Eq. 4.8]{Chen93}}]
If $G=\{x\rightarrow xy, y\rightarrow y\}$,
then
\begin{equation*}
D^n(x)=x\sum_{k=1}^n\Stirling{n}{k}y^k.
\end{equation*}
\end{proposition}

\begin{proposition}[{\cite[Section 2.1]{Dumont96}}]
If $G=\{x\rightarrow xy, y\rightarrow xy\}$,
then
\begin{equation*}
D^n(x)=x\sum_{k=0}^{n-1}\Eulerian{n}{k}x^{k}y^{n-k}.
\end{equation*}
\end{proposition}

\begin{proposition}[{\cite{Chen122}}]\label{Chen2}
If $G=\{x\rightarrow x^2y, y\rightarrow x^2y\}$,
then
\begin{equation*}
D^n(x)=\sum_{k=0}^{n-1}\left < \!\! \left < {n \atop k} \right > \!\! \right >x^{2n-k}y^{k+1}.
\end{equation*}
\end{proposition}

\begin{proposition}[{\cite{Ma12}}]
If $G=\{x\rightarrow y^2, y\rightarrow xy\}$,
then
$$D^n(x)=\sum_{k=0}^{\lrf{({n-1})/{2}}}W_{n,k}x^{n-2k-1}y^{2k+2},\quad D^n(y)=\sum_{k=0}^{\lrf{{n}/{2}}}W_{n,k}^{\textit{l}}x^{n-2k}y^{2k+1},$$
where $W_{n,k}$ is the number of permutations in $\msn$ with $k$ interior peaks and $W^{\textit{l}}_{n,k}$ is the number of permutations in $\msn$ with $k$ left peaks.
\end{proposition}

The purpose of this paper is to show that some combinatorial sequences, such as second-order Eulerian numbers and Eulerian numbers of type $B$, can be generated by context-free grammars.
\section{Results}\label{sec:result}
For $n\geq 0$, we always assume that
$$(xD)^{n+1}(x)=(xD)(xD)^n(x)=xD((xD)^n(x)).$$
The following theorem is in a sense ``dual" to Proposition~\ref{Chen2}.
\begin{theorem}
If $G=\{x\rightarrow xy, y\rightarrow xy\}$, then
\begin{equation*}
(xD)^n(x)=\sum_{k=0}^{n-1}\left < \!\! \left < {n \atop k} \right > \!\! \right >x^{2n-k}y^{k+1} \quad {\text for}\quad n\geq 1.
\end{equation*}
\end{theorem}
\begin{proof}
For $n\geq 1$, we define
\begin{equation}\label{Enk-def}
(xD)^n(x)=\sum_{k=0}^{n-1}E(n,k)x^{2n-k}y^{k+1}.
\end{equation}
Note that $$(xD)(x)=x^2y, (xD)(x^2y)=x^4y+2x^3y^2.$$
Then $E(1,0)=\left < \!\! \left < {1 \atop 0} \right > \!\! \right >=1, E(2,0)=\left < \!\! \left < {2 \atop 0} \right > \!\! \right >=1$ and $E(2,1)=\left < \!\! \left < {2 \atop 1} \right > \!\! \right >=2$.
Using~\eqref{Enk-def}, we obtain
$$(xD)(xD)^n(x)=\sum_{k=0}^{n-1}(2n-k)E(n,k)x^{2n-k+1}y^{k+2}+
\sum_{k=0}^{n-1}(k+1)E(n,k)x^{2n-k+2}y^{k+1}.$$
Therefore,
\begin{equation*}\label{Enk-recu}
E(n+1,k)=(2n-k+1)E(n,k-1)+(k+1)E(n,k).
\end{equation*}
Comparing with~\eqref{Eu-second-recu}, we see that
the coefficients $E(n,k)$ satisfy the same recurrence and initial conditions as $\left < \!\! \left < {n \atop k} \right > \!\! \right >$, so they agree.
\end{proof}

Now we present the main result of this paper.
\begin{theorem}
For $n\geq 1$, we have the the following results:
\begin{itemize}
\item [({$c_1$})]If $G=\{x\rightarrow xy^2,y\rightarrow x^2y\}$,
 then
\begin{equation*}\label{Bnk-der}
D^n(xy)=xy\sum_{k=0}^nB(n,k)x^{2n-2k}y^{2k}.
\end{equation*}
\item [({$c_2$})]If $G=\{x\rightarrow xy^2,y\rightarrow x^2y\}$,
 then
\begin{equation*}
D^n(x^2y^2)=2^nx^2y^2\sum_{k=0}^n\Eulerian{n+1}{k}x^{2n-2k}y^{2k}.
\end{equation*}
\item [({$c_3$})]If $G=\{x\rightarrow xy^2,y\rightarrow x^2y\}$,
 then
\begin{equation*}
D^n(x)=x\sum_{k=1}^nN(n,k)x^{2n-2k}y^{2k},
\end{equation*}
where the number $N({n,k})$ enumerates perfect matchings of $[2n]$ with the restriction that only $k$ matching pairs have odd smaller entries (see~\cite[A185411]{Sloane}).
   \item [({$c_4$})]If $G=\{x\rightarrow xy,y\rightarrow xy\}$,
then $$D^n(xy^r)=x\sum_{k=0}^{n}\Eulerian{n+r}{k}_{r}x^ky^{n+r-k}.$$
  \item [({$c_5$})]If $G=\{x\rightarrow xy^2,y\rightarrow xy\}$,
then $$D^n(x)=x\sum_{k=0}^{n-1}2^k\Eulerian{n}{k}x^ky^{2n-2k}.$$
  \item [({$c_6$})]Consider the numbers $T(n,k)$ with generating function
  $$\sqrt{A(2t,z)}=1+\sum_{n\geq 1}\sum_{k=1}^nT(n,k)t^k\frac{z^n}{n!},$$
where $A(t,z)$ is given by~\eqref{Atz} (see~\cite[A156920]{Sloane}. If $G=\{x\rightarrow xy^2,y\rightarrow xy\}$,
then $$D^n(y)=\sum_{k=1}^{n}T(n,k)x^ky^{2n-2k+1}.$$
\item [({$c_7$})]If $G=\{x\rightarrow x^2y,y\rightarrow y\}$,
then $$D^n(x)=x\sum_{k=1}^nk!{n \brace k}x^ky^k.$$
\item [({$c_8$})]If $G=\{x\rightarrow x^2y,y\rightarrow y^2\}$,
then $$D^n(x)=x\sum_{k=1}^nk!\stirling{n}{k}x^{k}y^n.$$
\item [({$c_9$})]If $G=\{x\rightarrow xy^2,y\rightarrow y^2\}$,
 then
\begin{equation*}
D^n(x)=x\sum_{k=1}^n\lah{n}{k}y^{n+k}.
\end{equation*}
\item [({$c_{10}$})]If $G=\{x\rightarrow xy^3,y\rightarrow y^3\}$,
then $$D^n(x)=x\sum_{k=1}^{n}b(n,k)y^{2n+k},$$
where $b(n,k)$ is the number of forests with $k$ rooted ordered trees with $n$ non-root vertices labeled in an organic way (see~\cite[A035342]{Sloane}).
\item [({$c_{11}$})]For a fixed positive integer $r\geq 4$, if $G=\{x\rightarrow xy^r,y\rightarrow y^r\}$,
then $$D^n(x)=x\sum_{k=1}^{n}a(n,k;r)y^{(r-1)n+k},$$
where $a(n,k;r)$ enumerates unordered $n$-vertex $k$-forests composed of $k$ plane increasing $r$-ary trees (see~\cite[A035469,A049029,A049385,A092082]{Sloane}).
\item [({$c_{12}$})]If $G=\{x\rightarrow x^2y,y\rightarrow xy\}$,
then $$D^n(y)=x^n\sum_{k=1}^nd(n,k)y^k,$$
where $d(n,k)$ is the number of increasing mobiles (circular rooted trees) with $n$ nodes and $k$ leaves (see~\cite[A055356]{Sloane}).
\end{itemize}
\end{theorem}
\begin{proof}
We only prove $(c_1)$ and the others can be proved in a similar way.
Note that $D(x)=xy^2$ and $D(y)=x^2y$. Then
$$D(xy)=xy(x^2+y^2),D^2(xy)=D(D(xy))=xy(x^4+6x^2y^2+y^4).$$
For $n\geq 1$, we define
\begin{equation*}
D^n(xy)=xy\sum_{k=0}^{n}G(n,k)x^{2n-2k}y^{2k}.
\end{equation*}
Hence $G(1,0)=B(1,0)$ and $G(1,1)=B(1,1)$.
Since
$$D^{n+1}(xy)=D(D^n(xy))=\sum_{k=0}^n(2n-2k+1)G(n,k)x^{2n-2k+1}y^{2k+3}+\sum_{k=0}^n(2k+1)G(n,k)x^{2n-2k+3}y^{2k+1},$$
there follows
\begin{equation*}
G(n+1,k)=(2n-2k+3)G(n,k-1)+(2k+1)G(n,k).
\end{equation*}
It follows from~\eqref{Bnk-Euleriannum} that
$G(n,k)$ satisfies the same recurrence and initial conditions as $B(n,k)$, so they agree.
\end{proof}



\end{document}